\pdfoutput=1
\documentclass[oupthm]{CUP-JNL-BCM-clean}

\usepackage{graphicx}
\usepackage{multicol,multirow}
\usepackage{amsmath,amssymb,amsfonts}
\usepackage{mathrsfs}
\usepackage{rotating}
\usepackage{appendix}
\usepackage[numbers]{natbib}

\theoremstyle{oupplain}
\newtheorem{theorem}{Theorem}[section]
\newtheorem{thm}{Theorem}[section]
\newtheorem{lemma}[theorem]{Lemma}
\newtheorem{modlemma}[theorem]{Lemma}
\newtheorem{cor}[theorem]{Corollary}
\theoremstyle{oupdefinition}

\theoremstyle{oupremark}

\theoremstyle{oupproof}
\newtheorem{proof}{Proof}

\numberwithin{equation}{section}

\articletype{RESEARCH ARTICLE}
\jname{Canadian Mathematical Society}
\jyear{2020}

\usepackage{mathtools}
\usepackage[labelformat=empty]{subcaption}
\usepackage{amsopn}
\usepackage{amssymb}
\usepackage{fixmath}
\usepackage{microtype}

\DeclareMathOperator{\rad}{rad}
\DeclareMathOperator{\li}{li}
\DeclareMathOperator{\ord}{ord}

\let\vectorsym\mathbold

\newcommand{\Z}{\mathbb{Z}}
\newcommand{\R}{\mathbb{R}}
\renewcommand{\e}{{\vectorsym{e}}}
\renewcommand{\b}{{\vectorsym{b}}}
\newcommand{\x}{{\vectorsym{x}}}

\renewcommand{\d}{\,\mathrm{d}}
\renewcommand{\L}{\overline{L}}

\let\brace\relax

\DeclarePairedDelimiter{\brace}{\lbrace}{\rbrace}
\DeclarePairedDelimiter{\abs}{\lvert}{\rvert}
\DeclarePairedDelimiter{\norm}{\lVert}{\rVert}
\DeclarePairedDelimiter{\normo}{\lVert}{\rVert_1}

\DeclarePairedDelimiter{\paren}{\lparen}{\rparen}
\NewDocumentCommand\set{somm}{\IfBooleanTF#1{\brace*{\,#3{}:{}#4\,}}{\brace[#2]{\,#3{}:{}#4\,}}}

\ifpdf
\hypersetup{
  pdftitle={A new lower bound in the $abc$ conjecture},
  pdfauthor={Curtis Bright}
}
\fi

\begin{document}

\begin{Frontmatter}

\title{A new lower bound in the $abc$ conjecture}

\author{Curtis Bright}

\authormark{C.~Bright}

\address{\orgname{School of Computer Science, University of Windsor, and Department of Mathematics and Statistics, Carleton University}, \email{cbright@uwindsor.ca}, webpage: \url{www.curtisbright.com}}



\abstract{We prove that there exist infinitely
many coprime numbers $a$, $b$, $c$ with $a+b=c$
and $c>\rad(abc)\exp(6.563\sqrt{\log c}/\log\log c)$.  These are
the most extremal examples currently known in the $abc$ conjecture, thereby
providing a new lower bound on the tightest possible form of the conjecture.
Our work builds on that of van Frankenhuysen (1999) who proved the existence of
examples satisfying the above bound with the constant $6.068$ in place of $6.563$.
We show that the constant $6.563$ may be replaced by
$4\sqrt{2\delta/e}$ where~$\delta$ is a constant such that
all unimodular lattices of sufficiently large dimension~$n$ contain a nonzero vector
with $\ell_1$ norm at most $n/\delta$.
}

\end{Frontmatter}

\section{Introduction}
Three natural numbers $a$, $b$, $c$ are said to be an \emph{$abc$ triple}
if they do not share a common factor and satisfy the equation
\[ a + b = c . \]
Informally, the $abc$ conjecture says that
large $abc$ triples cannot be `very composite',
in the sense of $abc$ having a prime factorization containing large powers of small primes.
The \emph{radical} of $abc$ is defined
to be the product of the primes in the prime factorization of $abc$, i.e.,
\[ \rad(abc) \coloneqq \prod_{p\mid abc}p . \]
The $abc$ conjecture then states that $abc$ triples satisfy%
\begin{equation} c = O\paren[\big]{\rad(abc)^{1+\epsilon}} \label{abc} \end{equation}
for every $\epsilon>0$, where the implied big-$O$ constant may depend on $\epsilon$.

Presently, the conjecture is far from being proved;
not a single $\epsilon$ is known for which~\eqref{abc} holds.\footnote{A proof of the $abc$ conjecture is claimed by S.~Mochizuki, but this has not been
accepted by the general mathematical community.~\cite{scholze2018abc}}
The best known upper bound is due to C.~L.~Stewart and K.~Yu~\cite{stewart2001} and says that $abc$ triples satisfy
\[ c = O\paren[\big]{\exp(\rad(abc)^{1/3}(\log\rad(abc))^3)} . \]
On the other hand, Stewart and Tijdeman~\cite{Stewart1986} proved in 1986 that there are infinitly many $abc$ triples
with
\begin{equation} c > \rad(abc)\exp\paren[\big]{\kappa\sqrt{\log c}/\log\log c} \label{abc2} \end{equation}
for all $\kappa<4$.
Such $abc$ triples are exceptional in the sense that their radical is relatively small in comparison to $c$
and they provide a lower bound on the best possible form of~\eqref{abc}.
In 1997, van Frankenhuysen~\cite{VanFrankenhuysen200091} improved this lower bound by showing that~\eqref{abc2}
holds for $\kappa=4\sqrt{2}$, and in 1999 he improved this to $\kappa=6.068$ using a sphere-packing idea credited to H.~W.~Lenstra, Jr.
We improve this further by showing that there are infinitely many $abc$ triples
satisfying~\eqref{abc2} with $\kappa=6.563$.

\section{Preliminaries}
Let $S$ be a set of prime numbers.
An $S$-unit is defined to be a rational number whose numerator and denominator in lowest terms are divisible by only the primes in $S$.
That is, one has
\[ \text{$S$-units} \coloneqq \set[\bigg]{\pm\prod_{p_i\in S}p_i^{e_i}}{e_i\in\Z} . \]
This generalizes the notion of units of $\Z$; in particular, the $\emptyset$-units are $\pm1$.
The \emph{height} of a rational number $p/q$ in lowest terms is $h(p/q)\coloneqq\max\brace{\abs{p},\abs{q}}$.
This provides a convenient way of measuring the `size' of an $S$-unit.
Finally, if $\x=(x_1,\dotsc,x_n)$ is a vector in $\R^n$, we let
\[ \norm{\x}_k \coloneqq \paren[\Big]{\sum_{i=1}^n\abs{x_i}^k}^{1/k} \]
be its standard $\ell_k$ norm.
The existence of exceptional $abc$ triples follows from some basic results in the geometry of numbers
along with estimates for prime numbers provided by the prime number theorem.
In particular, we rely on a result of Rankin~\cite{MR0027296} guaranteeing
the existence of a short nonzero vector in a suitably chosen lattice.
\subsection{The odd prime number lattice}
The result involves in an essential way the \emph{odd prime number lattice} $L_n$ generated by the rows $\b_1$, $\dotsc$, $\b_n$ of the matrix
\[ \begin{bmatrix}\b_1\\\b_2\\\b_3\\\vdots\\\b_n\end{bmatrix} = \begin{bmatrix}
\log3 &  &  &  & & \log3 \\
 & \log5 &  &  & & \log5 \\
 &  & \log7 & & & \log7 \\
 &  & & \ddots & & \vdots \\
 &  & & & \log p_n & \log p_n
\end{bmatrix}
\]
where $p_i$ denotes the $i$th odd prime number.
This lattice has a number of interesting applications.
For example, it is used in Schnorr's factoring algorithm~\cite{Schnorr}
and Micciancio's proof that approximating
the shortest vector to within a constant factor is NP-hard under a randomized reduction~\cite{Micciancio}.
There is an obvious isomorphism between the points of $L_n$ and the positive $\brace{p_1,\dotsc,p_n}$-units given by
\[ \sum_{i=1}^n e_i\b_i\leftrightarrow\prod_{i=1}^n p_i^{e_i} . \]
Furthermore, this relationship works well with a natural notion of size, as shown in the following lemma.
\begin{lemma}\label{heightlem}
$\normo{\x} = 2 \log h\paren{p/q}$ where\/
$\x=\sum_{i=1}^n e_i\b_i$ and\/ $p/q=\prod_{i=1}^n p_i^{e_i}$ is expressed in lowest terms.
\end{lemma}
\begin{proof}
Without loss of generality suppose $p\geq q$.  Then
\[ \normo{\x} = \sum_{i=1}^n\abs[\big]{e_i\log p_i} + \abs[\Big]{\sum_{i=1}^n e_i\log p_i} 
= \log p + \log q + \log p - \log q 
= 2\log p \]
as required, since $h(p/q)=p$ by assumption.
\end{proof}%

\subsection{The kernel sublattice}
Let $P$ be the set of positive $\brace{p_1,\dotsc,p_n}$-units, and consider the map $\phi$ reducing the elements of $P$ modulo $2^m$.
Since each $p_1$, $\dotsc$, $p_n$ is odd, $\phi\colon P\to(\Z/2^m\Z)^*$ is well-defined.
The odd prime number lattice $L_n$ has an important sublattice that we call the \emph{kernel sublattice} $L_{n,m}$.  It consists of
those vectors whose associated $\brace{p_1,\dotsc,p_n}$-units lie in the kernel of $\phi$.  Formally, we define
\[ L_{n,m} \coloneqq \set[\bigg]{\sum_{i=1}^n e_i\b_i}{\prod_{i=1}^n p_i^{e_i} \equiv 1\pmod{2^m}} . \]
Figure~\ref{kernelfig} plots the first two coordinates of vectors in the
kernel sublattice for varying~$m$.

\begin{figure}%
\centering
\begin{subfigure}[b]{0.24\textwidth}
\includegraphics[width=\textwidth]{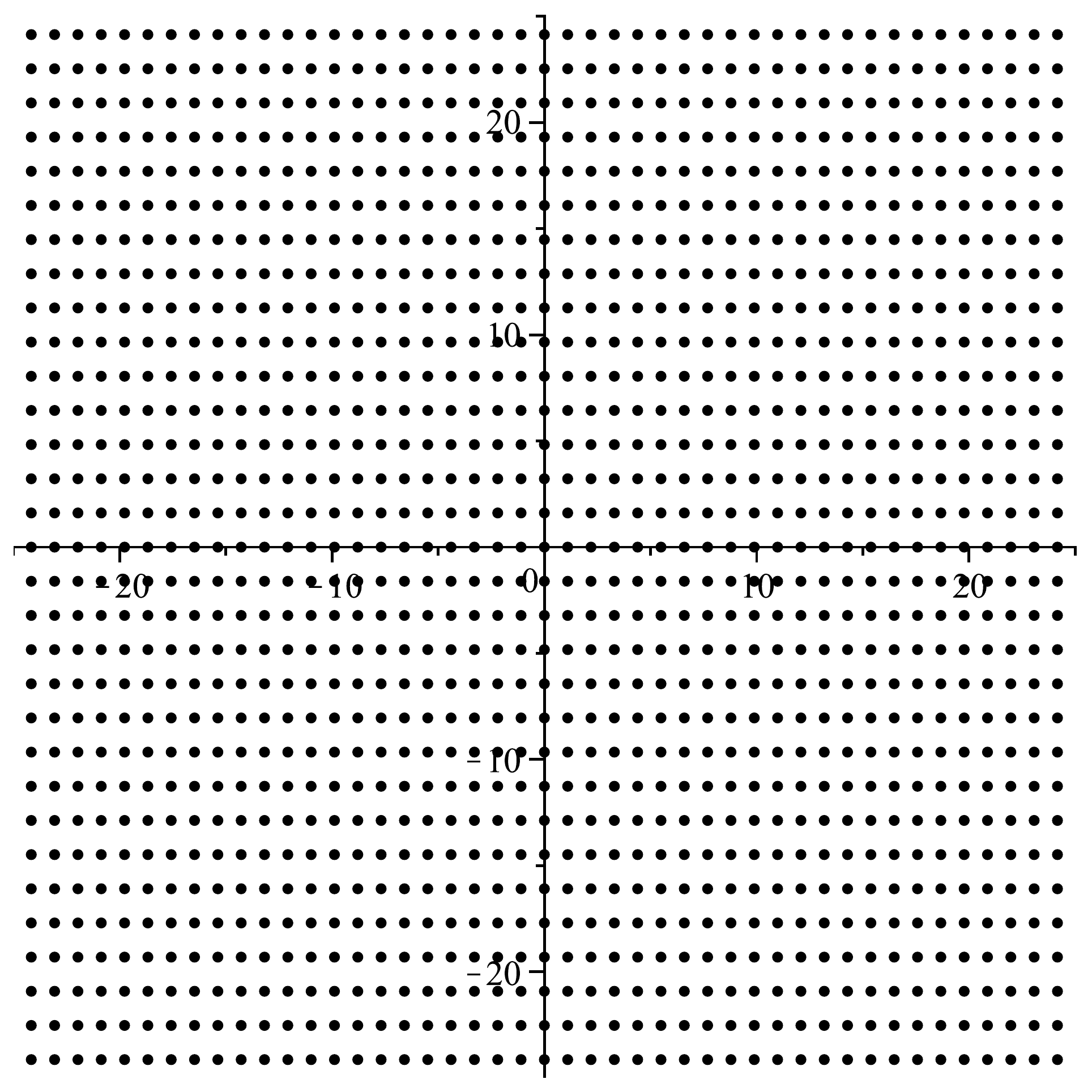}
\end{subfigure}
\begin{subfigure}[b]{0.24\textwidth}
\includegraphics[width=\textwidth]{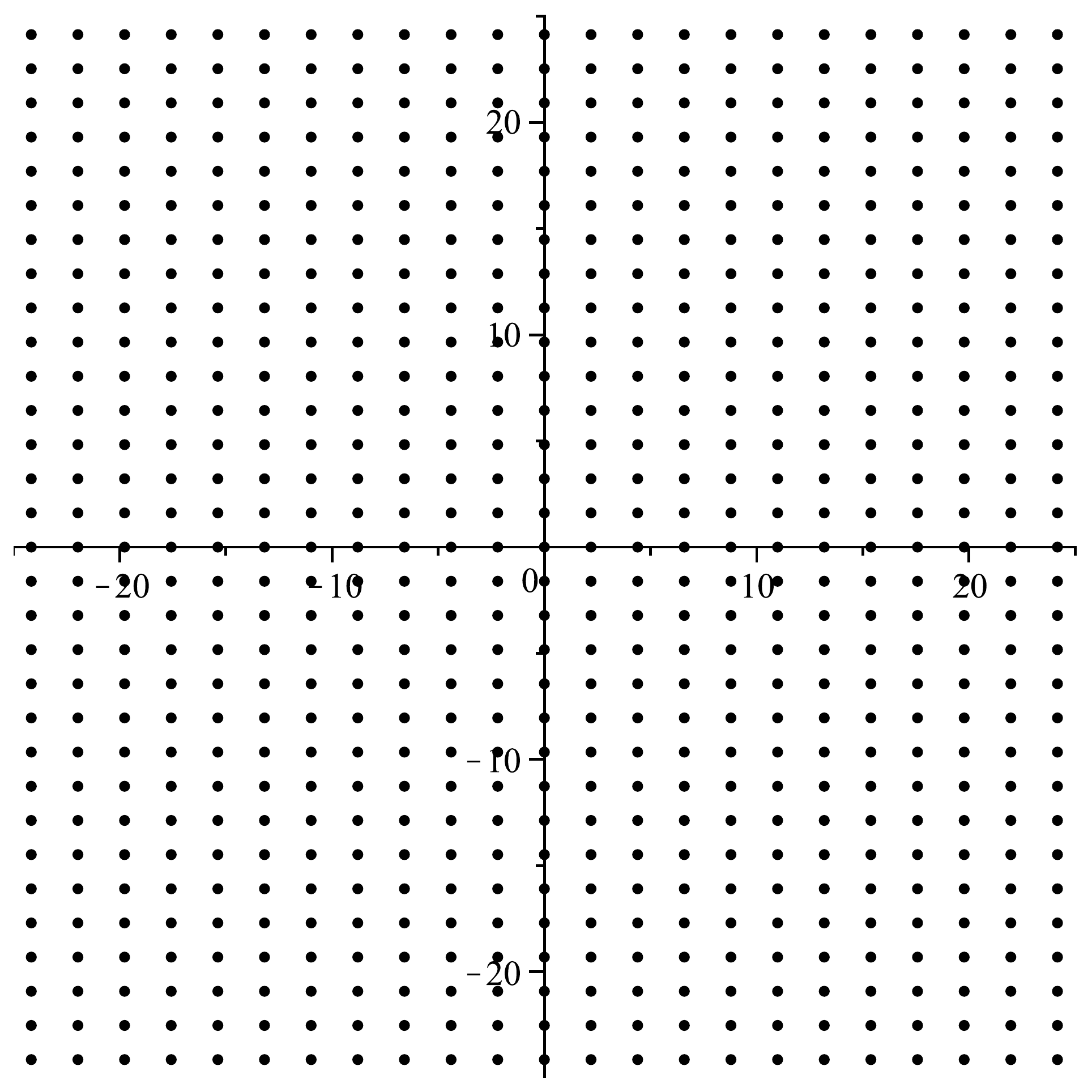}
\end{subfigure}
\begin{subfigure}[b]{0.24\textwidth}
\includegraphics[width=\textwidth]{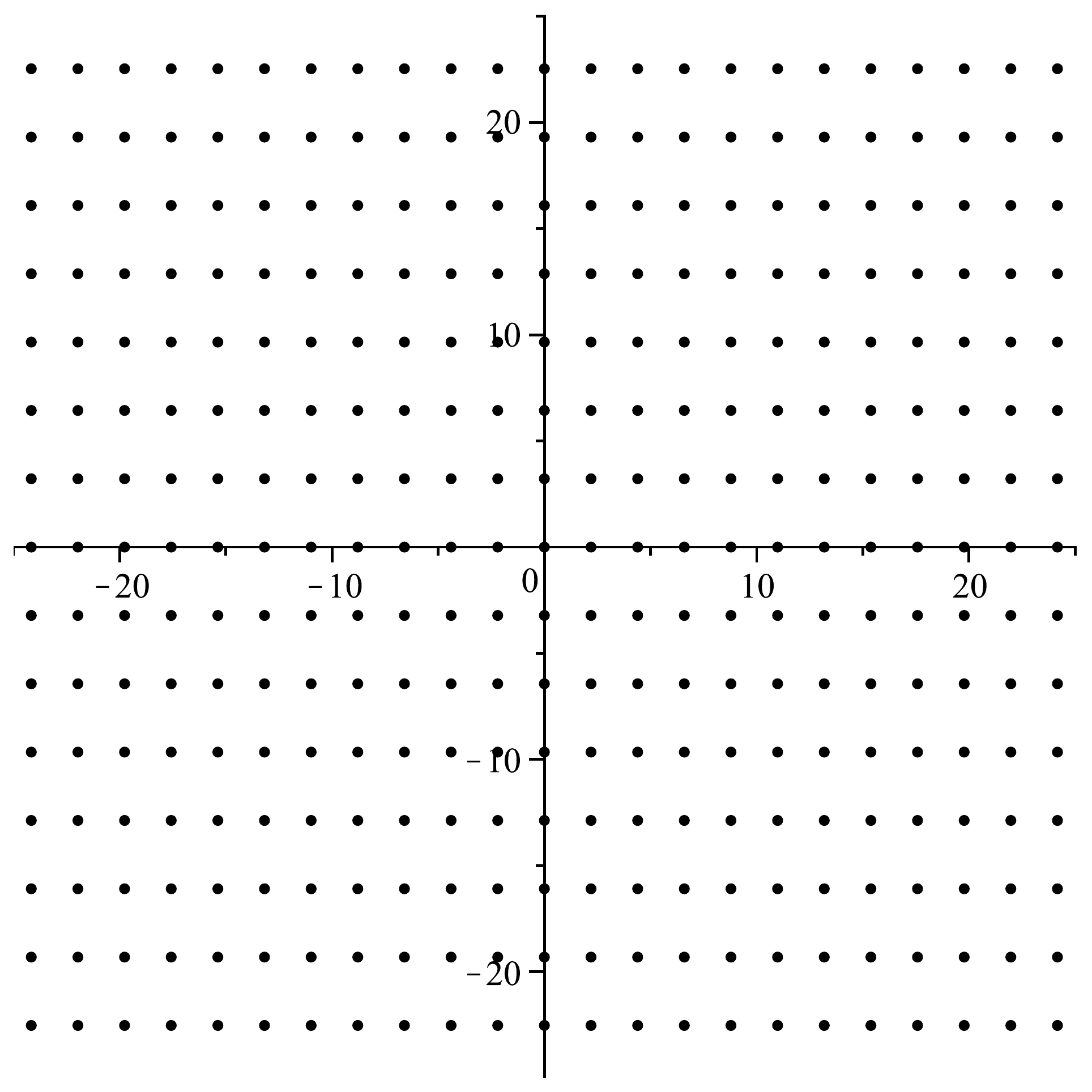}
\end{subfigure}
\begin{subfigure}[b]{0.24\textwidth}
\includegraphics[width=\textwidth]{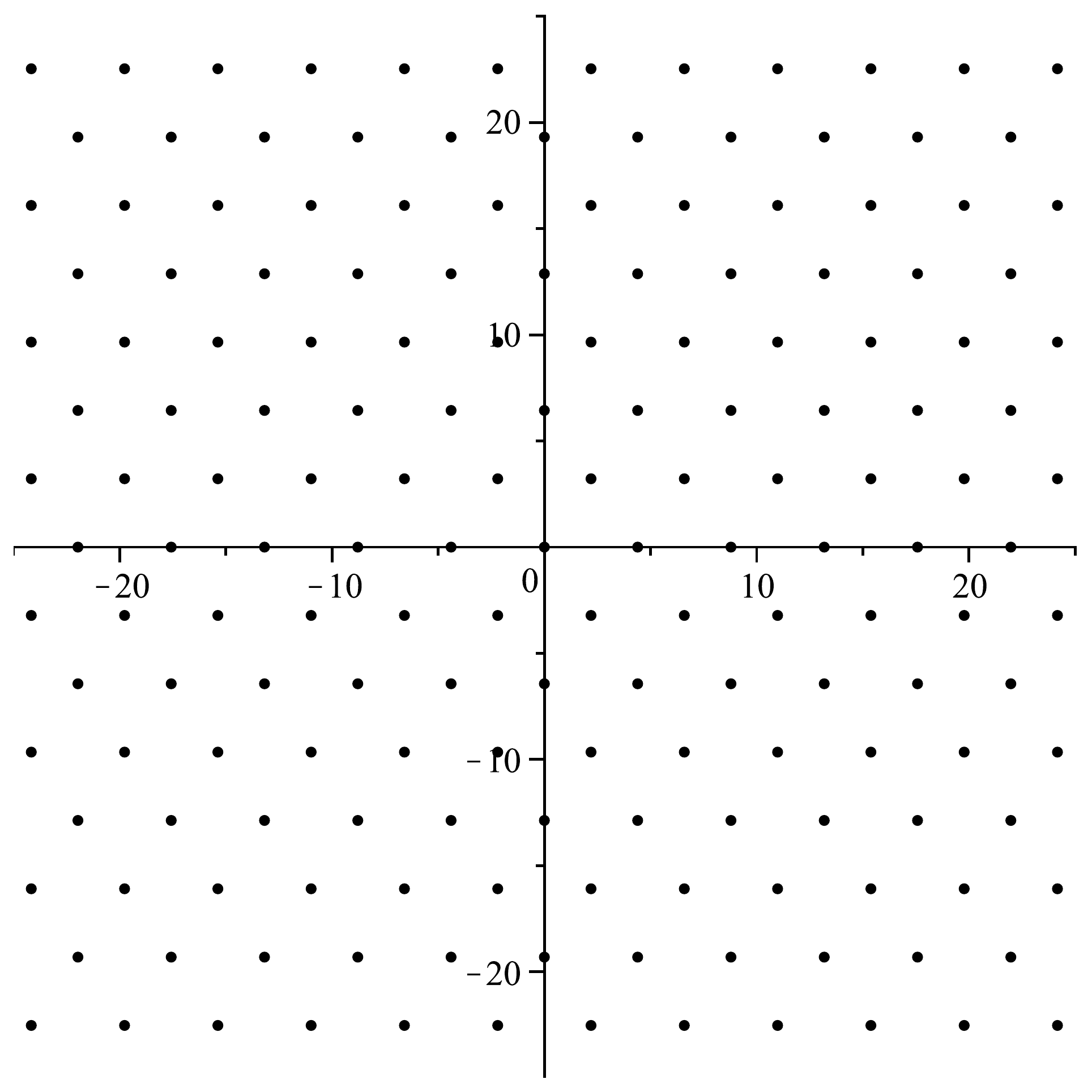}
\end{subfigure}
\begin{subfigure}[b]{0.24\textwidth}
\includegraphics[width=\textwidth]{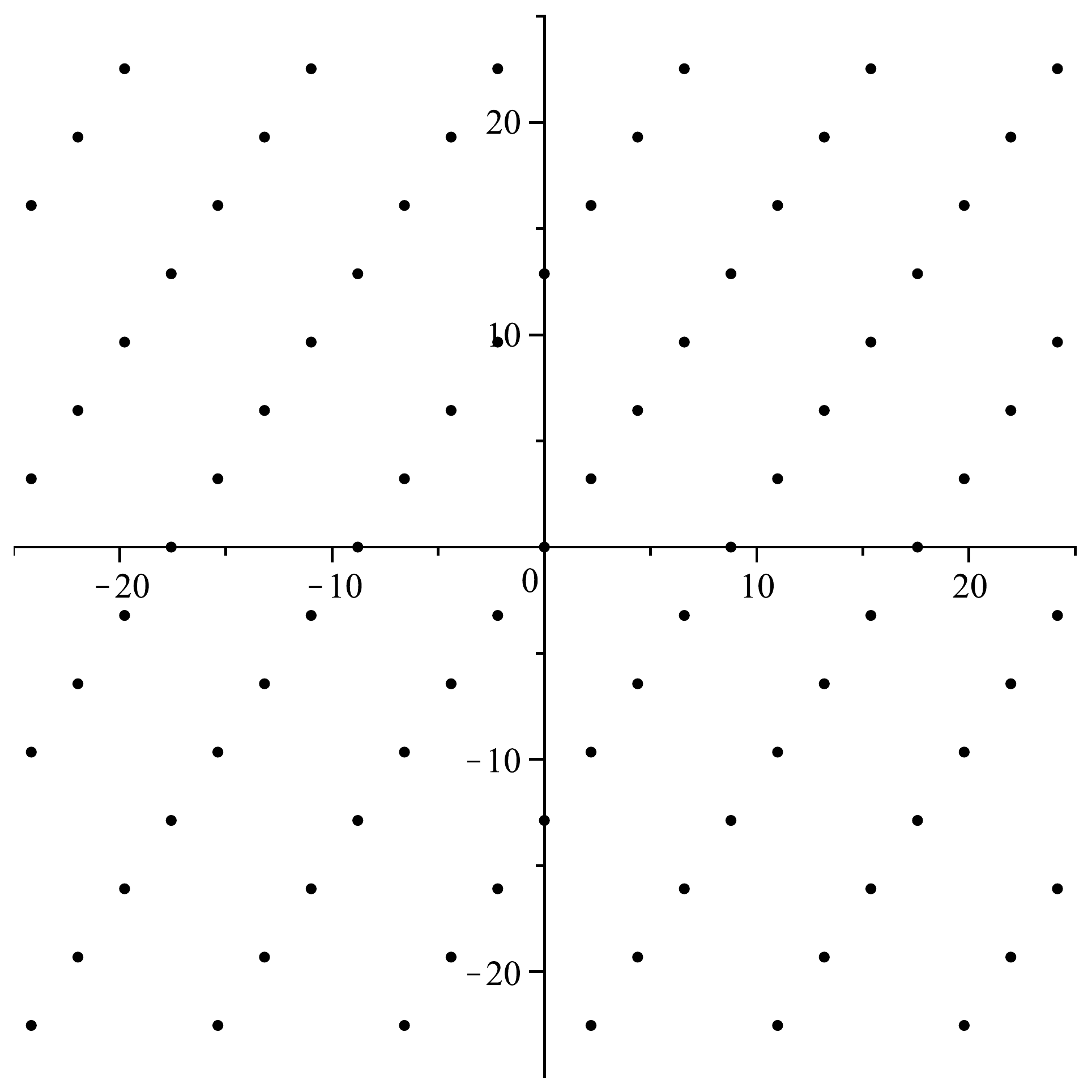}
\end{subfigure}
\begin{subfigure}[b]{0.24\textwidth}
\includegraphics[width=\textwidth]{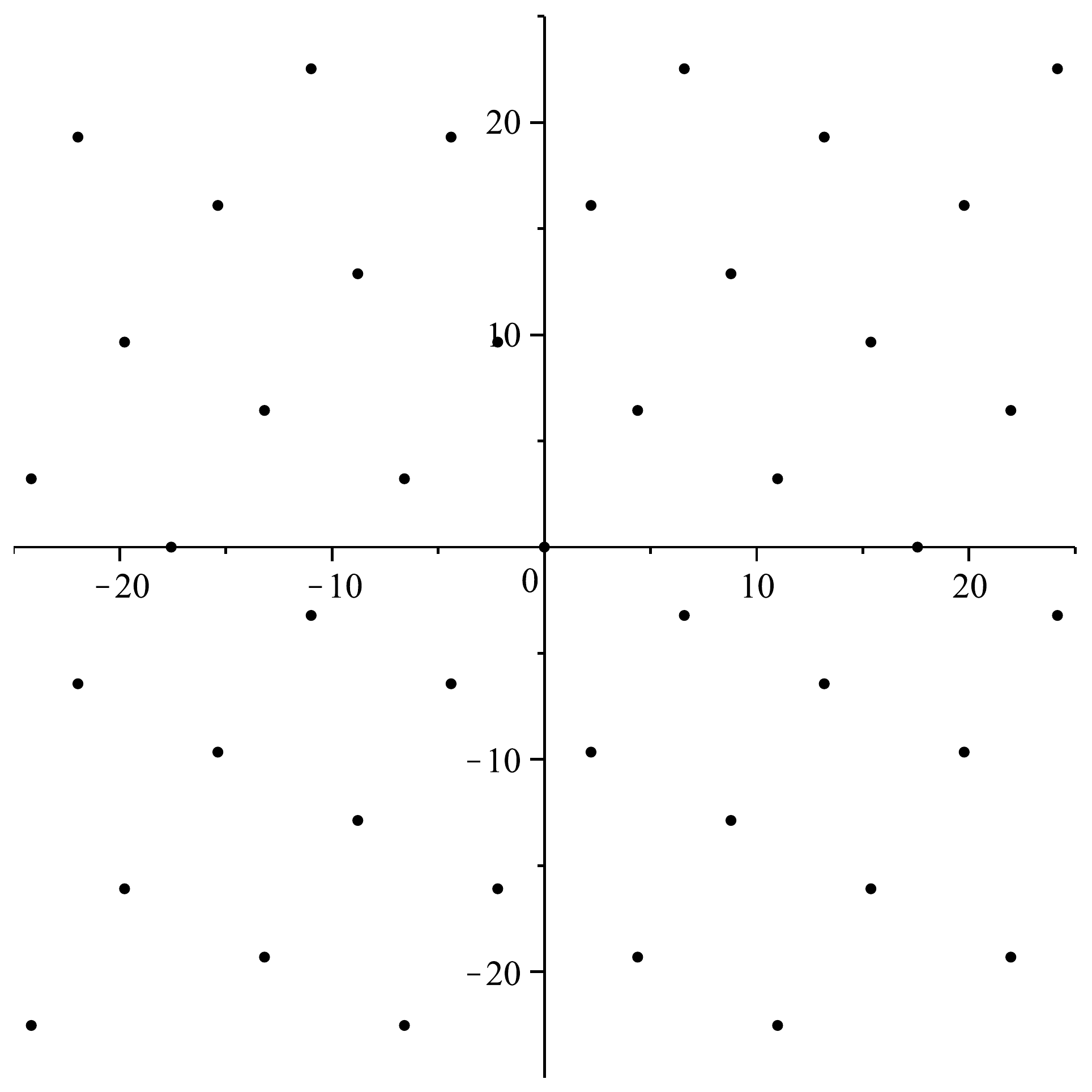}
\end{subfigure}
\begin{subfigure}[b]{0.24\textwidth}
\includegraphics[width=\textwidth]{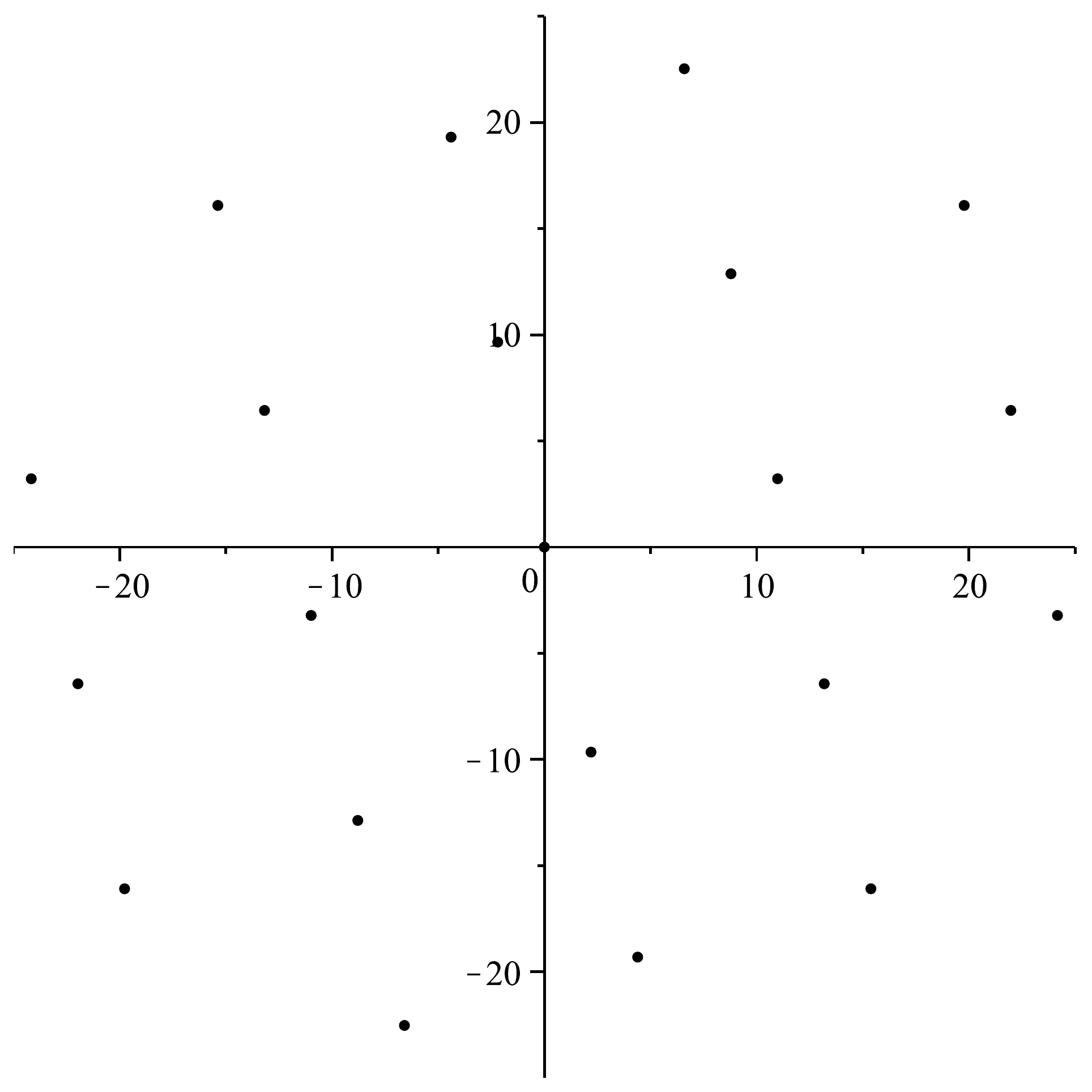}
\end{subfigure}
\begin{subfigure}[b]{0.24\textwidth}
\includegraphics[width=\textwidth]{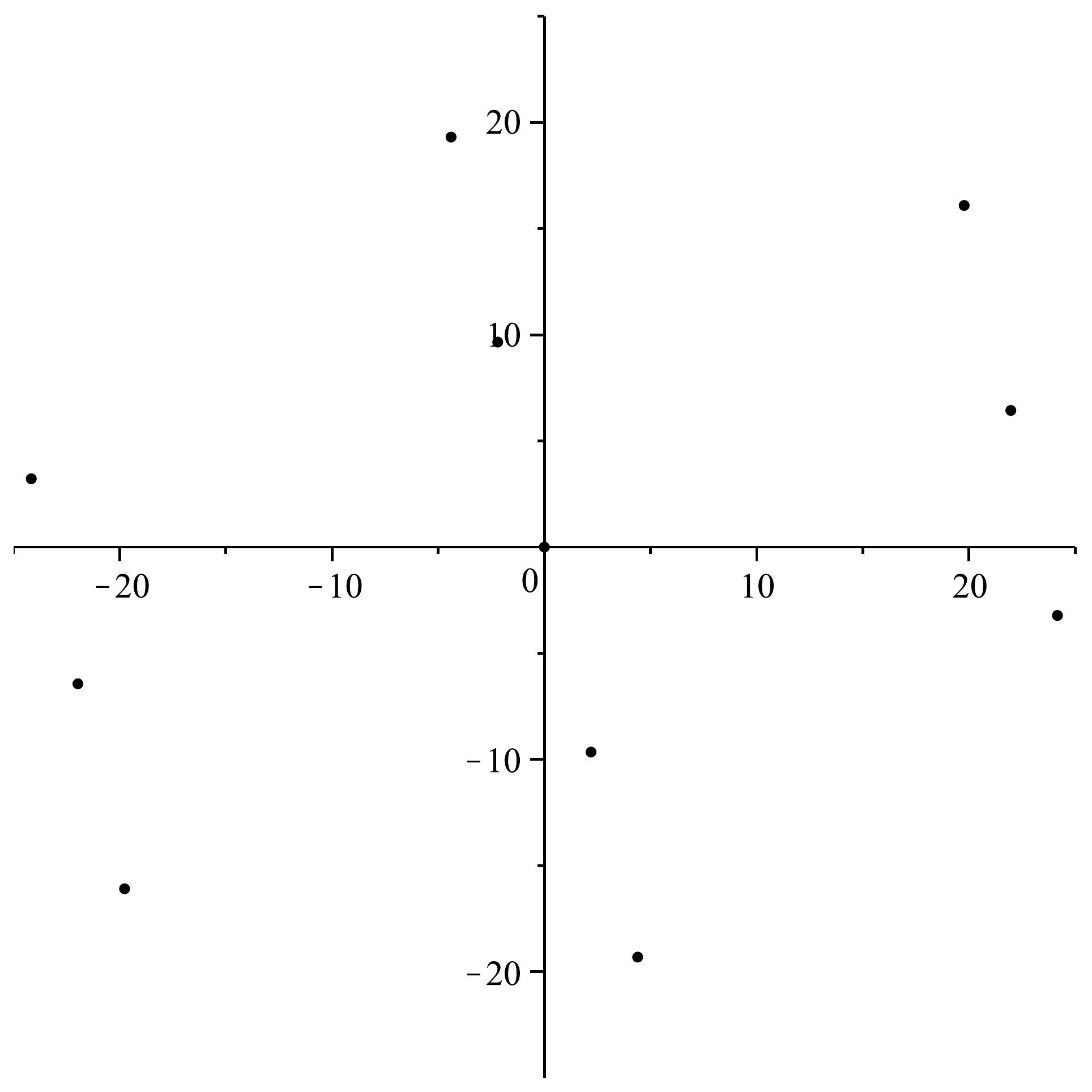}
\end{subfigure}
\caption{Plots of $\set{(x,y)}{(x,y,z)\in L_{2,m}}$ for $1\leq m\leq 8$.}\label{kernelfig}
\end{figure}

\begin{lemma}\label{indexlem}
$L_{n,m}$ is a sublattice of\/ $L_n$ of index\/ $2^{m-1}$ when\/ $n\geq2$.
\end{lemma}
\begin{proof}
Note that $L_{n,m}$ is discrete and closed under addition and subtraction.
$L_{n,m}$ also contains the $n$ linearly independent vectors $\ord_{2^m}(p_i)\b_i$ for $1\leq i\leq n$, so this demonstrates that $L_{n,m}$ is a full-rank sublattice of $L_n$.

Since $3$ and $5$ generate $(\Z/2^m\Z)^*$, when $n\geq2$ we have $\phi(P)=(\Z/2^m\Z)^*$.
Since $L_n\cong P$ and $L_{n,m}\cong\ker\phi$ it follows that
$L_n/L_{n,m} \cong (\Z/2^m\Z)^*$
by the first isomorphism theorem.  Thus the index of $L_{n,m}$ in $L_n$ is $\abs{(\Z/2^m\Z)^*}=2^{m-1}$.
\end{proof}

\subsection{Hermite's constant}\label{sec:hermite}
The \emph{Hermite constant}\/ $\gamma_n$ is defined to be the smallest positive number such that
every lattice of dimension~$n$ and volume $\det(L)$
contains a nonzero vector~$\x$ with
\[ \norm{\x}^2_2 \leq \gamma_n \det(L)^{2/n} . \]
We are interested in the ``Manhattan distance'' $\ell_1$ norm instead of the usual Euclidean norm, so we define the related constants $\delta_n$ by the
smallest positive number such that every full-rank lattice of dimension $n$ contains a nonzero vector $\x$ with
\[ \normo{\x} \leq \delta_n \det(L)^{1/n} . \]
By Minkowski's theorem~\cite{cassels2012introduction} applied to a generalized octahedron (a `sphere' in the $\ell_1$ norm),
every full-rank lattice of dimension $n$ contains a nonzero lattice point $\x$ with
$\normo{\x} \leq (n!\det(L))^{1/n}$.
It follows that $\delta_n \leq (n!)^{1/n} \sim n/e$,
but better bounds on $\delta_n$ are known. Blichfeldt~\cite{blich36} showed that
\[ \delta_n \leq \sqrt{\frac{4(n+1)(n+2)}{3\pi(n+3)}}\paren[\bigg]{\frac{2(n+1)}{n+3}\paren[\Big]{\frac{n}{2}+1}!}^{1/n} \sim \frac{n}{ \sqrt{1.5\pi e}} , \]
where $x!\coloneqq\Gamma(x+1)$.  Improving this, Rankin~\cite{MR0027296} showed the following.
\begin{lemma}\label{rankinlem}For all integer $n$ and real $x\in[1/2,1]$, we have
\[ \delta_n \leq \paren[\Big]{\frac{2-x}{1-x}}^{x-1}\paren[\Big]{\frac{1+x n}{x}(x n)!}^{1/n}\frac{n^{1-x}}{x!} \sim \paren[\Big]{\frac{2-x}{1-x}}^{x-1}\paren[\Big]{\frac{x}{e}}^x \frac{n}{x!} . \]
\end{lemma}
\begin{cor}\label{hermlem}
Let $\delta$ be a constant such that $\delta_n\leq n/\delta+O(\log n)$.
Then a permissible value for $\delta$ is $\max\limits_{1/2\leq x\leq 1}\paren[\big]{\frac{1-x}{2-x}}^{x-1}\paren[\big]{\frac{e}{x}}^x x!\approx 3.65931$.
\end{cor}
\begin{proof}
Note that $((1+xn)/x)^{1/n}=1+O((\log n)/n)$ and
\[ (x n)!^{1/n} = \paren[\Big]{\sqrt{2\pi xn}\mspace{1.5mu}\paren[\Big]{\frac{xn}{e}}^{xn}\paren[\big]{1+O(n^{-1})}}^{1/n} = \paren[\Big]{\frac{xn}{e}}^x\paren[\big]{1+O\paren[\big]{\tfrac{\log n}{n}}} . \]
Then by Lemma~\ref{rankinlem} it follows that
\[ \delta_n \leq \paren[\Big]{\frac{2-x}{1-x}}^{x-1}\paren[\Big]{\frac{x}{e}}^x \frac{n}{x!} + O(\log n), \]
and the function $x\mapsto\paren[\big]{\frac{1-x}{2-x}}^{x-1}\paren[\big]{\frac{e}{x}}^x x!$ for $1/2\leq x\leq 1$
reaches a maximum of approximately $3.65931$ at $x\approx0.645467$.
\end{proof}

The best possible value $\delta$ can achieve in Corollary~\ref{hermlem} is unknown,
but the Minkowski--Hlawka theorem~\cite{cassels2012introduction} applied to an generalized octahedron
shows that in any dimension~$n$ there is always a full-rank lattice $L$ with
all of its nonzero lattice points $\x$ having $\normo{\x}>(\zeta(n)\mspace{1.5mu}n!\det(L))^{1/n}\mspace{-1.5mu}/2$;
here $\zeta$ is the Riemann zeta function.
It follows that $\delta_n>(\zeta(n)\mspace{1.5mu}n!)^{1/n}\mspace{-1.5mu}/2\sim n/(2e)$, so
we must have $\delta\leq 2e$.

\subsection{A full-rank kernel sublattice}

Since $L_{n,m}\in\R^{n+1}$ is of dimension $n$ (i.e., not full-rank) it is awkward to
use Rankin's result on $L_{n,m}$ directly.
The basis matrix of $L_{n,m}$ cannot simply be rotated to embed it in $\R^n$, since
rotation does not preserve the $\ell_1$ norm.  To circumvent this and work with
a full-rank lattice we adjoin the new basis vector $\b_{n+1}=[0,\dotsc,0,n^3]$
to $L_n$ to form a full-rank lattice $\L_n$ (and similarly a full-rank lattice $\L_{n,m}$).

\begin{lemma}\label{vollem2}
The volume of\/ $\L_{n,m}$ is\/ $2^{m-1}n^3\prod_{i=1}^n\log p_i$ when\/ $n\geq2$.
\end{lemma}
\begin{proof}
The basis matrix of $L_n$ adjoined with $\b_{n+1}$ is an upper-triangular matrix, so
$\det(\L_n)=n^3\prod_{i=1}^n\log p_i$.  The index of $\L_{n,m}$ in $\L_n$ is $2^{m-1}$
when $n\geq2$ by the same argument as in Lemma~\ref{indexlem}, so
$\det(\L_{n,m})=2^{m-1}\det(\L_n)$.
\end{proof}

Our choice of $m$ will ultimately be asymptotic to $n\log_2 n$,
and in this case $\det(\L_{n,m})^{1/(n+1)}$ grows slightly more than linearly in~$n$.
\begin{lemma}\label{detlem}
If\/ $m\sim n\log_2 n$ then\/ $\det(\L_{n,m})^{1/(n+1)}=O(n^{1+\epsilon})$ for all\/ $\epsilon>0$.
\end{lemma}
\begin{proof}%
Lemma~\ref{vollem2} implies
$\det(\L_{n,m})^{1/(n+1)} < 2^{m/n} n^{3/n} \paren[\big]{\prod_{i=1}^n\log p_i}^{1/n}$.
Note that $m/n=\log_2 n + o(\log_2 n)<(1+\epsilon)\log_2 n$ for all $\epsilon>0$ and sufficiently large $n$.
Thus $2^{m/n}<n^{1+\epsilon}$ for sufficiently large~$n$, and
the remaining factors are $O(n^\epsilon)$ since $n^{3/n}=O(1)$ and $\paren[\big]{\prod_{i=1}^n\log p_i}^{1/n}<\log p_n = O(\log n)$.
\end{proof}

Finally, we will require the fact that any vector in $\L_n$ including a nontrivial coefficient on
$\b_{n+1}$ must be sufficiently large (have length at least $n^3$ in the $\ell_1$ norm).

\begin{lemma}\label{sizelem}
If\/ $\x=\sum_{i=1}^{n+1}e_i\b_i$ then\/ $\normo{\x}\geq n^3\abs{e_{n+1}}$.
\end{lemma}
\begin{proof}
We have $\normo{\x}=\sum_{i=1}^n\abs{e_i}\log p_i+\abs[\big]{\sum_{i=1}^n e_i\log p_i+e_{n+1}n^3}$.

Without loss of generality suppose that $e_{n+1}>0$ and for contradiction suppose $\normo{\x}<n^3 e_{n+1}$.  Then
\[ \sum_{i=1}^n e_i\log p_i+e_{n+1}n^3 \leq \abs[\Big]{\sum_{i=1}^n e_i\log p_i+e_{n+1}n^3} < n^3 e_{n+1} - \sum_{i=1}^n \abs{e_i}\log p_i \]
implies $\sum_{i=1}^n (e_i+\abs{e_i}) \log p_i < 0$, and this is nonsensical since the left-hand side is nonnegative.
\end{proof}

\subsection{Asymptotic formulae}
Let $x\coloneqq p_n$ and let $\pi(x)$ be the prime counting function, so that $n=\pi(x)-1$.
The prime number theorem~\cite{ingham1990distribution} states that $\pi(x)\sim\li(x)$
where $\li(x)$ is the logarithmic integral $\int_0^x\frac{\d t}{\log t}$ with asymptotic expansion
\begin{equation} \li(x) = \frac{x}{\log x} + \frac{x}{\log^2 x} + \frac{2x}{\log^3 x} + O\paren[\Big]{\frac{x}{\log^4 x}} . \label{pnt} \end{equation}
In fact, the error term $\pi(x)-\li(x)$ is $O(x/\exp(C\log^{1/2} x))$ for some constant $C>0$.
The following estimates are consequences of this
(cf.~\cite[Lemma 2]{Stewart1986}).
For the convenience of the reader, proofs are given in the appendix.

\begin{lemma}\label{loglem}
$\sum_{i=1}^n\log p_i = n\log p_n - n - p_n/\log^2 p_n + O(p_n/\log^3 p_n)$.
\end{lemma}

\begin{lemma}\label{logloglem}
$\sum_{i=1}^n\log\log p_i = n\log\log p_n - p_n/\log^2 p_n + O(p_n/\log^3 p_n)$.
\end{lemma}

\section{Exceptional \boldmath{$abc$} triples}
For our purposes the importance of the kernel sublattice is that it lets us show the existence of $abc$ triples in which $c$ is large relative to $\rad(abc)$.
The following lemma shows how this may be done.
\begin{lemma}\label{abclem}
For all\/ $m\lesssim n\log_2 n$ and sufficiently large\/ $n$, there exists an\/ $abc$ triple satisfying
\[ \frac{2^{m-1}}{\prod_{i=1}^n p_i}\rad(abc) \leq c \quad\text{and}\quad 2 \log c \leq \frac{n+O(\log n)}{\delta} \paren[\Big]{2^{m-1}n^3\prod_{i=1}^n\log p_i}^{1/(n+1)} . \]
\end{lemma}
\begin{proof}
By the definition of $\delta$ from Corollary~\ref{hermlem}, for all sufficiently large~$n$ there exists a nonzero $\x\in\L_{n,m}$ with
\begin{equation} \normo{\x} \leq \paren[\Big]{\frac{n+1}{\delta}+O(\log n)} \det\paren{\L_{n,m}}^{1/(n+1)}
. \label{star} \end{equation}
Say $\x=\sum_{i=1}^{n+1} e_i\b_i$.  For sufficiently large $n$ we must have $e_{n+1}=0$, since
by Lemma~\ref{sizelem} if $e_{n+1}\neq0$ then $\normo{\x}\geq n^3$.  This would contradict~\eqref{star}
since by Lemma~\ref{detlem} the right-hand side is $O(n^{2+\epsilon})$.

Let $\prod_{i=1}^n p_i^{e_i}=p/q$ be expressed in lowest terms.
By construction of the kernel sublattice, we have that
$p/q \equiv 1 \pmod{2^m}$.
Let $c\coloneqq h(p/q)=\max\brace{p,q}$, $b\coloneqq\min\brace{p,q}$, and $a\coloneqq c-b$, so that $a$, $b$, $c$ form an $abc$ triple.  Furthermore, we see that
\[ c \equiv b \pmod{2^m} \]
so that $c=b+k2^m$ for some positive integer $k\leq c/2^m$.  Note $a$ is divisible by~$2$ and any other prime that divides it also divides $k$, so that $\rad(a)\leq2k\leq c/2^{m-1}$.
Furthermore, by construction of $b$ and $c$, $\rad(bc)\leq\prod_{i=1}^n p_i$
and the first bound follows.
The second bound follows from~\eqref{star} and Lemmas~\ref{heightlem} and~\ref{vollem2}.
\end{proof}

\subsection{Optimal choice of \boldmath{$m$}}
The first bound in Lemma~\ref{abclem} allows us to show the existence of infinitely many
$abc$ triples whose ratio of $c$ to $\rad(abc)$ grows arbitrarily large.
Using the second bound, we can even show that this ratio grows faster than a function of $c$.
It is not immediately clear how to choose $m$ optimally, i.e., to maximize the ratio $c/\rad(abc)$.

For convenience, let $R$ denote the right-hand side of the second inequality in Lemma~\ref{abclem} with $l_n\coloneqq O(\log n)$.
Then $2^{m-1}=\paren[\big]{\frac{\delta R}{n+l_n}}^{n+1}/(n^3\prod_{i=1}^n\log p_i)$, so the bounds of Lemma~\ref{abclem} can be rewritten in terms of $R$:
\begin{equation}
\frac{(\delta R/(n+l_n))^{n+1}}{n^3\prod_{i=1}^n p_i\log p_i}\rad(abc) \leq c \quad\text{and}\quad 2\log c \leq R . \label{abcfirst}
\end{equation}
The question now becomes how to choose $R$ in terms of~$n$ so that $c/\rad(abc)$ is maximized.

Taking the logarithm of the first inequality in~\eqref{abcfirst} gives
\[ (n+1)\log\paren[\Big]{\frac{\delta R}{n+l_n}} - 3\log n - \sum_{i=1}^n \log p_i - \sum_{i=1}^n \log\log p_i + \log\rad(abc) \leq \log c . \]
Using the asymptotic formulae in Lemmas~\ref{loglem} and~\ref{logloglem} with $\log(n+l_n)=\log n+O(l_n/n)$, this becomes
\begin{equation} n\log\paren[\Big]{\frac{e\delta R}{np_n\log p_n}} + \frac{2p_n}{\log^2 p_n} + O\paren[\Big]{\frac{p_n}{\log^3 p_n}} + \log\rad(abc) \leq \log c . \label{five} \end{equation}
By the prime number theorem $n=\li(p_n)+O(p_n/\log^2 p_n)$ and~\eqref{pnt}
the leftmost term becomes
\[ n\log\paren[\bigg]{\frac{e\delta R}{p_n^2\paren[\big]{1+1/\log p_n+O(1/\log^2 p_n)}}} , \]
and with $\log(1+1/x)=1/x+O(1/x^2)$ as $x\to\infty$, this is
\[ n\log\paren[\Big]{\frac{e\delta R}{p_n^2}} - \frac{n}{\log p_n} + O\paren[\Big]{\frac{n}{\log^2 p_n}}
. \]
Using~\eqref{pnt} again on the last two terms and putting this back into~\eqref{five}, we get
\begin{equation} n\log\paren[\Big]{\frac{e\delta R}{p_n^2}} + \frac{p_n}{\log^2 p_n} + O\paren[\Big]{\frac{p_n}{\log^3 p_n}} + \log\rad(abc) \leq \log c , \label{six} \end{equation}
and our goal becomes to choose $R$ as a function of $n$ to maximize $n\log(e\delta R/p_n^2)$.
Choosing $R$ as asymptotically slow-growing as possible in terms of~$n$ will maximize this in terms of~$R$.
We must take $R>p_n^2/(e\delta)$ for the logarithm to be positive, so we take
$R\coloneqq kp_n^2$ for some constant $k$.
Note that with this choice $m\sim n\log_2 n$, so Lemma~\ref{abclem} applies.  We have that
$n\log(e\delta R/p_n^2)$ simplifies to
\[ n\log(e\delta k) \sim \frac{p_n}{\log p_n}\log(e\delta k) = \frac{\sqrt{R/k}}{\log\sqrt{R/k}}\log(e\delta k) \sim \frac{2\sqrt{R/k}}{\log R}\log(e\delta k) . \]

For fixed $R$ this is maximized when $k\coloneqq e/\delta$.
Using $R=ep_n^2/\delta$ in~\eqref{six},
\[ 2n + \frac{p_n}{\log^2 p_n} + O\paren[\Big]{\frac{p_n}{\log^3 p_n}} + \log\rad(abc) \leq \log c . \]
By the prime number theorem and~\eqref{pnt} again,
\[ \frac{2p_n}{\log p_n} + \frac{3p_n}{\log^2 p_n} + O\paren[\Big]{\frac{p_n}{\log^3 p_n}} + \log\rad(abc) \leq \log c . \]
Rewriting in terms of $R$,
\[ \frac{2\sqrt{\delta R/e}}{\log \sqrt{\delta R/e}} + \frac{3\sqrt{\delta R/e}}{\log^2 \sqrt{\delta R/e}} + O\paren[\Big]{\frac{\sqrt{R}}{\log^3 R}} + \log\rad(abc) \leq \log c . \]
Simplifying,
\[ \frac{4\sqrt{\delta R/e}}{\log(\delta R/e)} + \frac{12\sqrt{\delta R/e}}{\log^2(\delta R/e)} + O\paren[\Big]{\frac{\sqrt{R}}{\log^3 R}} + \log\rad(abc) \leq \log c . \]
Using $1/(x+y)=1/x-y/x^2+O(x^{-3})$ as $x\to\infty$ this gives
\[ \frac{4\sqrt{\delta R/e}}{\log (R/2)} + \frac{(12-4\log(2\delta/e))\sqrt{\delta R/e}}{\log^2 R} + O\paren[\Big]{\frac{\sqrt{R}}{\log^3 R}} + \log\rad(abc) \leq \log c . \]
Using that $2\delta<e^4$ the second term on the left is positive, and so for sufficiently large~$R$ the middle two terms are necessarily positive.
Therefore for sufficiently large~$R$ this can be simplified to
\[ \frac{4\sqrt{\delta R/e}}{\log(R/2)} + \log\rad(abc) \leq \log c . \]
Using that $2\log c\leq R$ from~\eqref{abcfirst} and the increasing monotonicity of $\sqrt{R}/\log(R/2)$ for sufficiently large $R$, we finally achieve that
\[ \frac{4\sqrt{2(\delta/e)\log c}}{\log\log c} + \log\rad(abc) \leq \log c . \]
Taking the exponential, this proves the following theorem.
\begin{thm}
There are infinitely many\/ $abc$ triples satisfying
\[ \exp\paren[\bigg]{\frac{4\sqrt{2(\delta/e)\log c}}{\log\log c}} \rad(abc) \leq c . \]
\end{thm}
Using the permissible value for $\delta$ derived by Rankin's bound in Corollary~\ref{hermlem},
the constant in the exponent becomes approximately $6.56338$.
As mentioned in Section~\ref{sec:hermite}, the best known upper bound on $\delta$ is $2e$,
meaning that the constant in the exponent would become $8$ if this upper bound was shown to be tight.

\section*{Acknowledgments}
The author would like to thank the reviewer for their detailed review
and useful feedback they provided on the first draft of this paper.

\bibliographystyle{siamplain}
\bibliography{references}

\printaddress

\newpage
\section*{Appendix}

\setcounter{section}{2}

\setcounter{theorem}{7}

\begin{modlemma}%
$\sum_{i=1}^n\log p_i = n\log p_n - n - p_n/\log^2 p_n + O(p_n/\log^3 p_n)$.
\end{modlemma}
\begin{proof}
Let $x\coloneqq p_n$, so the prime number theorem (with error term) gives $n=\li(x)+O(x/\log^4 x)$.
Rearranging the
asymptotic expansion of the logarithmic integral~\eqref{pnt} gives
\begin{align*}
x &= n\log x - \frac{x}{\log x} - \frac{2x}{\log^2 x} + O\paren[\Big]{\frac{x}{\log^3 x}} \\
&= n\log x - n - \frac{x}{\log^2 x} + O\paren[\Big]{\frac{x}{\log^3 x}} .
\end{align*}
An alternate form of the prime number theorem is $x=\sum_{p\leq x}\log p+O(x/\log^3 x)$, so
the left-hand side may be replaced by $\sum_{i=1}^n \log p_i$ from which the result follows.
\end{proof}

\begin{modlemma}%
$\sum_{i=1}^n\log\log p_i = n\log\log p_n - p_n/\log^2 p_n + O(p_n/\log^3 p_n)$.
\end{modlemma}
\begin{proof}
By Abel's summation formula with $f(k)\coloneqq\log\log k$ and
\[ a_k\coloneqq \begin{cases}
1 & \text{if $k$ is an odd prime} \\
0 & \text{otherwise}
\end{cases} \]
for $k$ up to $x\coloneqq p_n$, we have
\[ \sum_{i=1}^n\log\log p_i = n\log\log x - \int_2^x\frac{\pi(t)-1}{t\log t}\d t . \]
We have $\pi(t)-1=t/\log t+O(t/\log^2 t)$ by the prime number theorem,
so that
\[ \int_2^x\frac{\pi(t)-1}{t\log t}\d t = \int_2^x\frac{\d t}{\log^2 t} + O\paren[\bigg]{\int_2^x\frac{\d t}{\log^3 t}} . \]
The first integral on the right works out to
\[ \int_2^x\frac{\d t}{\log^2 t} = \li(x) - \frac{x}{\log x} + O(1) = \frac{x}{\log^2 x} + O\paren[\Big]{\frac{x}{\log^3 x}} \]
by the asymptotic expansion of the logarithmic integral.  The second integral on the right can split in two (around $\sqrt x$) and then estimated by
\[ \int_2^{\sqrt x}\frac{\d t}{\log^3 t}+\int_{\sqrt x}^x\frac{\d t}{\log^3 t} \leq \frac{\sqrt x}{\log^3 2}+\frac{x-\sqrt x}{\log^3\sqrt{x}} = O\paren[\Big]{\frac{x}{\log^3 x}} . \]
Putting everything together gives
\[ \sum_{i=1}^n\log\log p_i = n\log\log x - \frac{x}{\log^2 x} + O\paren[\Big]{\frac{x}{\log^3 x}} . \tag*{\qedsymbol}
\]\let\qedsymbol\relax
\end{proof}

\end{document}